\documentclass[11pt,a4paper]{amsart}

\RequirePackage[OT1]{fontenc}
\RequirePackage{amsthm,amsmath}
\RequirePackage[numbers]{natbib}
\RequirePackage[colorlinks,citecolor=blue,urlcolor=blue]{hyperref}

\usepackage{graphicx}
\newtheorem{theorem}{Theorem}
\newtheorem{lemma}[theorem]{Lemma}
\newtheorem*{lemma*}{Lemma}

\newtheorem{corollary}[theorem]{Corollary}

\newtheorem{definition}[theorem]{Definition}
\newtheorem{remark}[theorem]{Remark}

\newtheorem*{fact*}{Fact}

\usepackage[utf8]{inputenc}
\usepackage[english]{babel}
\usepackage{latexsym}
\usepackage{amssymb}
\usepackage{amsmath}

\newcommand{\Z}{\mathbb{Z}}
\newcommand{\R}{\mathbb{R}}
\newcommand{\C}{\mathbb{C}}
\newcommand{\E}{\mathbb{E}}

\begin{document}

\title[Linear statistics of the circular $\beta$-ensemble]{Linear statistics of the circular $\beta$-ensemble, Stein's method, and circular Dyson Brownian motion}

\author[C. Webb]{Christian Webb}
\address{Department of mathematics and systems analysis, Aalto University, PO Box 11000, 00076 Aalto, Finland}
\email{christian.webb@aalto.fi}
\date{\today}

\begin{abstract}
We study the linear statistics of the circular $\beta$-ensemble with a Stein's method argument, where the exchangeable pair is generated through circular Dyson Brownian motion. This generalizes previous results obtained in such a way for the CUE and provides a novel approach for studying linear statistics of $\beta$-ensembles. This approach allows studying simultaneously a collection of linear statistics whose number grows with the dimension of the ensemble. Also this approach requires estimating only low order moments of the linear statistics.
\end{abstract}

\maketitle

\section{Introduction}

The goal of this note is to study linear statistics of the circular $\beta$-ensemble (which we will usually denote by C$\beta$E or C$\beta$E$(n)$ if we wish to stress the dimension). More precisely, if $(e^{ix_1},...,e^{ix_n})$ is a realization of the $n$-dimensional C$\beta$E, we shall study the Wasserstein-1 distance of the law of

\begin{equation}
T_d=\left(\sum_{j=1}^n e^{ikx_j}\right)_{k=1}^d
\end{equation}

\noindent to the law of 

\begin{equation}
G_d=\left(\sqrt{\frac{2}{\beta}j}Z_j\right)_{j=1}^d,
\end{equation}

\noindent where $Z_j$ are i.i.d. standard complex Gaussians. 

\vspace{0.3cm}

Our main result will be that if $d$ grows slowly enough with $n$, the distance goes to zero as $n\to\infty$. Our approach will be to apply Stein's method for which we shall generate an exchangeable pair through circular Dyson Brownian motion. The estimates one will then need to apply Stein's method involve some low order moments of $T_d$ for which we can make use of results of \cite{jm}.

\vspace{0.3cm}

The motivation for this approach comes from \cite{fulman, ds}, where a similar approach is used for $\beta=2$ (as well as the circular real ensemble and circular quaternion ensemble, i.e. the Haar measure on the orthogonal and symplectic groups), though the relevant dynamics is interpreted through the heat kernel on the unitary group which does not generalize so obviously to other values of $\beta$. 

\vspace{0.3cm}

While the fact that finite collections of such linear statistics converge jointly in law to independent Gaussians with suitable variances, is certainly known (e.g. the approach of \cite{johansson} should be easily adapted to the circular case and more recently such a result is proven in \cite{jm} - for other work related to the linear statistics of the C$\beta$E, see e.g. \cite{spohn,fw}), what our approach offers is a rate of convergence (which is likely to be extremely far from the true one - in the case of CUE the rate is known to be superexponential, see \cite{johansson2} - which is much faster than the one our approach suggests) as well as a possibility to study the joint convergence of linear statistics whose number grows with $n$. Another benefit of this approach is that one only needs to estimate only rather few moments. To the author's knowledge, such results aren't known for C$\beta$E. Moreover, this approach through Stein's method coupled with Dyson Brownian motion has potential to be applied to other $\beta$-ensembles.

\vspace{0.3cm}

The outline of this note is the following: we begin by recalling the definition of the C$\beta$E and the relevant Wasserstein distance as well as stating our main result. Next we shall recall the approach in \cite{ds} for multivariate complex normal approximation,  the definition of circular Dyson Brownian motion, and point out what the relevant estimates we shall need for applying Stein's method to our case. These estimates involve the generator of circular Dyson Brownian motion acting on certain power sums, which are simple to calculate exactly, along with moment bounds of power sums which can be estimated with results from \cite{jm}. Finally we point out as an application of the results of \cite{jm} a limit theorem for the logarithm of the characteristic polynomial of the C$\beta$E. This is very similar to a result of \cite{hko} for the CUE. 

\vspace{0.3cm}

{\bf{Acknowledgements:}} The author wants to thank two anonymous referees for helpful comments about the article, as well as K. Kyt\"ol\"a for useful discussions. This work was supported by the Academy of Finland.
 
\section{The circular $\beta$ ensemble, The Wasserstein distance, and the main result}

The purpose of this section is to state our main result and to do this, we recall the definition of the C$\beta$E and the Wasserstein-1 distance. 

\vspace{0.3cm}

\begin{definition}
Let 

\begin{equation}
\Delta_n=\lbrace (x_1,...,x_n)\in[0,2\pi]^n: x_1\leq x_2\leq ...\leq x_n\rbrace
\end{equation}

\noindent and $\beta>0$. The $n$-dimensional C$\beta$E is the following probability measure on $\Delta_n$: 

\begin{equation}
\frac{n!}{Z_{n,\beta}}\prod_{j<k}|e^{ix_j}-e^{ix_k}|^\beta \prod_{j=1}^n\frac{dx_j}{2\pi},
\end{equation}

\noindent where the normalization constant is a Selberg integral and can be evaluated exactly:

\begin{equation}
Z_{n,\beta}=\int_{[0,2\pi]^n}\prod_{j<k}|e^{ix_j}-e^{ix_k}|^\beta \prod_{j=1}^n\frac{dx_j}{2\pi}=\frac{\Gamma\left(1+n\frac{\beta}{2}\right)}{\Gamma\left(1+\frac{\beta}{2}\right)^n}.
\end{equation}

\end{definition}

\begin{remark}
We will often identify $[0,2\pi)$ with the unit circle and $\Delta_n$ with a subset of the $n$-fold product of the unit circle with itself.
\end{remark}

The Wasserstein-1 distance is a metric on the space of random variables taking values in a  fixed underlying space (which we'll take to be Euclidean, but more general cases are possible) with finite first absolute moment. Convergence with respect to it is equivalent to convergence in law along with convergence of the first absolute moment. Let us recall its two equivalent definitions (see e.g. Chapter 6 in \cite{villani} for more information on Wasserstein distances):

\begin{definition}
The Wasserstein-1 distance between the laws of two $\R^d$ (or $\C^d$ as we'll actually be interested in) valued random variables - $X$ and $Y$ - is 

\begin{equation}
\mathcal{W}_1^{(d)}(X,Y)=\inf \E(|X-Y|),
\end{equation}

\noindent where the infimum is over all couplings of $X$ and $Y$.

\vspace{0.3cm}

An equivalent definition for the metric (a result due to Kantorovich and Rubinstein) is given by  

\begin{align}
\notag \mathcal{W}_1^{(d)}(X,Y)=\sup\lbrace &\left.\E(f(X)) -\E(f(Y))\right| \ f:\R^d\to \R, \\
&|f(x)-f(y)|\leq |x-y|\ \mathrm{for\ all\ } x,y\in \R^d\rbrace.
\end{align}
\end{definition}

We can now state our main result. 

\begin{theorem}\label{th:main}
Let $(e^{ix_j})_{j=1}^n$ be drawn from the $n$-dimensional $\mathrm{C}\beta\mathrm{E}$ with $\beta>0$, $d=\mathit{o}(n^{\frac{2}{7}})$, 

\begin{equation}
T_d=\left(\sum_{j=1}^n e^{ikx_j}\right)_{k=1}^d
\end{equation}

\noindent and

\begin{equation}
G_d=\left(\sqrt{\frac{2}{\beta}j}Z_j\right)_{j=1}^d,
\end{equation}

\noindent where $Z_j$ are i.i.d. standard complex Gaussians. Then 

\begin{equation}
\mathcal{W}_1^{(d)}(T_d,G_d)=\mathcal{O}(d^{7/2}/n)
\end{equation}

\noindent as $n\to\infty$.
\end{theorem}

\begin{remark}
As in \cite{ds}, we could consider instead of $T_d$ a vector of the form 

\begin{equation}
\left(\sum_{j=1}^n e^{ikx_j}\right)_{k=r}^d,
\end{equation}

\noindent where also $r$ grows with $n$ and one will get constraints on $r$ and $d$ for the vanishing of the Wasserstein distance with similar methods as those we use. For simplicity, we only consider the case of $T_d$. 

\end{remark}

\begin{remark}
One can use this result to study linear statistics of functions on the unit circle with nice enough regularity by Fourier expanding them and applying our result. 
\end{remark}

\section{Stein's method and circular Dyson Brownian motion}

We'll give a short informal sketch of the Stein's method argument for multivariate normal approximation that will be relevant for us. For a detailed treatment, see e.g. \cite{meckes}. After this, we shall state the precise theorem (that appears in \cite{ds}) that we shall make use of. Next we shall review the definition and some basic properties of circular Dyson Brownian motion and how it ties into our Stein's method argument. 

\vspace{0.3cm}

\subsection{Stein's method} For simplicity we'll consider the case of real normal variables (the complex one follows from this). Let us assume that $\Sigma$ is a symmetric positive definite $d\times d$ matrix. We'll denote by $Y$ a $d$-dimensional vector of i.i.d. standard Gaussians and by $Y_\Sigma$ we denote $\sqrt{\Sigma}Y$.

\vspace{0.3cm}

We'll also use the following notation: by $\langle\cdot,\cdot\rangle_{\mathrm{HS}}$ we denote the Hilbert-Schmidt inner product of two matrices

\begin{equation}
\langle A,B\rangle_{\mathrm{HS}}=\mathrm{Tr}(AB^*),
\end{equation}

\noindent where $B^*$ denotes the Hermitian conjugate of $B$. We'll denote by $\Vert\cdot \Vert_{\mathrm{HS}}$ the corresponding norm.

\vspace{0.3cm}

We will then make use of the following facts (see \cite{meckes})

\begin{fact*}[Fact 1]
A random $d$-dimensional vector $X$ agrees in law with $Y_\Sigma$ if 

\begin{equation}
\E\left(\langle \mathrm{Hess}f(X),\Sigma\rangle_{\mathrm{HS}}-\langle X,\nabla f(X)\rangle\right)=0
\end{equation}

\noindent for each $f\in C^2(\R^d)$ for which the above integrand is in $L^1$ (with respect to the randomness). Here $\mathrm{Hess}f$ is the Hessian matrix of $f$, and the second inner product is the Euclidean inner product of $\R^d$.
\end{fact*}

\begin{fact*}[Fact 2]
If $g\in \C^\infty(\R^d)$, then 

\begin{equation}
h(x)=U_o g(x):=\int_0^1\frac{1}{2t}(\E g(\sqrt{t}x+\sqrt{1-t}Y_\Sigma)-\E g(Y_\Sigma))dt
\end{equation}

\noindent is a solution to the differential equation

\begin{equation}
\langle x,\nabla h(x)\rangle -\langle \mathrm{Hess}\ h(x),\Sigma\rangle_{\mathrm{HS}}=g(x)-\E g(Y_\Sigma).
\end{equation}

\end{fact*}

Let us now assume that we have a random vector $X$ for which we wish to show that the law of $X$ is close to that of $Y_\Sigma$ in the sense of the Wasserstein distance, and let us further assume that we have another random vector $X'$ on the same probability space as $X$ and $X'\stackrel{d}{=}X$. Moreover, let us assume that 

\begin{equation}
\E(X'-X|X)=-\Lambda X+V,
\end{equation}

\noindent for some invertible deterministic matrix $\Lambda$ and some random vector $V$. We'll want to think of $X'$ being close to $X$ so that when for example Taylor expanding $f(X')$ around $X$ for some function $f$, we can ignore high enough order terms. Also we assume that 

\begin{equation}
\E((X'-X)(X'-X)^T|X)=2\Lambda \Sigma+M,
\end{equation}

\noindent where $\Sigma$ is again our deterministic symmetric positive definite matrix and $M$ is a random $d\times d$-dimensional matrix.

\vspace{0.3cm}

Let us fix some $g\in C^\infty(\R^d)$ and let $f=U_o g$. Then as $X\stackrel{d}{=}X'$

\begin{align}
0&=\frac{1}{2}\E\left(\langle\Lambda^{-1}(X'-X),\nabla f(X')+\nabla f(X)\rangle\right)\notag \\
&=\frac{1}{2}\E\left(\langle\Lambda^{-1}(X'-X),\nabla f(X')-\nabla f(X)\rangle\right)+\E\left(\langle \Lambda^{-1}(X'-X),\nabla f(X)\rangle\right)\\
&=\notag \frac{1}{2}\E\left(\langle\Lambda^{-1}(X'-X),\mathrm{Hess}f(X) (X'-X)\rangle\right)+\E\left(\langle \Lambda^{-1}(X'-X),\nabla f(X)\rangle\right)\\
\notag &\qquad +...,
\end{align}

\noindent where we Taylor expanded $\nabla f(X')-\nabla f(X)$ around $X$ and $...$ denotes higher order terms in the expansion. Noting that

\begin{equation}
\langle\Lambda^{-1}(X'-X),\mathrm{Hess}f(X) (X'-X)\rangle=\langle \Lambda^{-1}(X'-X)(X'-X)^T,\mathrm{Hess} f(X)\rangle_{\mathrm{HS}}
\end{equation}

\noindent and conditioning on $X$, we find that 

\begin{align}
\notag 0&=\E(\mathrm{Hess}f(X),\Sigma\rangle_{\mathrm{HS}}-\E(\langle X,\nabla f(X)\rangle)\\
&\qquad  +\frac{1}{2}\E(\langle \Lambda^{-1}M,\mathrm{Hess}f(X)\rangle_{\mathrm{HS}})+\E(\langle \Lambda^{-1}V,\nabla f(X)\rangle)+...
\end{align}

From Fact 2, we then find

\begin{equation}
\E g(X)-\E g(Y_\Sigma)=\frac{1}{2}\E(\langle \Lambda^{-1}M,\mathrm{Hess}f(X)\rangle_{\mathrm{HS}})+\E(\langle \Lambda^{-1}V,\nabla f(X)\rangle)+...
\end{equation}

As $\mathrm{Hess}f$ and $\nabla f$ are bounded, if we can control $\Lambda^{-1}M$ and $\Lambda^{-1}V$ (and the higher order terms), this suggests that we can control the Wasserstein distance. This is indeed the case. We'll actually construct a one parameter family of the vectors $X'$ through Dyson Brownian motion started from an independent C$\beta$E realization and the closeness of $X$ and $X'$ will come from the $t\to 0$ limit. Let us state the actual theorem for the Stein's method argument in the following form (see Theorem 1.3 in \cite{ds} and Theorem 4 in \cite{meckes} for proofs)

\begin{theorem}[D\"obler and Stoltz, Meckes]\label{th:stein}
Let $W,W_t$ (for $t>0$) be $\C^d$ valued $L^2(\mathbb{P})$ random vectors on the same probability space $(\Omega,\mathcal{A},\mathbb{P})$ such that for each $t>0$, $(W,W_t)\stackrel{d}{=}(W_t,W)$. Let $Z\in\C^d$ be a $d$-dimensional random vector whose entries are i.i.d. standard complex Gaussians. Suppose that there exist non-random matrices $\Lambda,\Sigma\in \C^{d\times d}$ such that $\Lambda$ is invertible and $\Sigma$ is positive definite. Assume further that there exists a random vector $R\in \C^d$, random matrices $S,T\in \C^{d\times d}$, and a deterministic function $s:(0,\infty)\to \R$ with the following properties 

\begin{align*}
\mathrm{(i)} & \frac{1}{s(t)}\E\left(W_t-W|W\right)\stackrel{t\to 0}{\to}-\Lambda W+R \ \mathrm{in} \ L^1(\mathbb{P})\\
\mathrm{(ii)} & \frac{1}{s(t)}\E\left((W_t-W)(W_t-W)^*|W\right)\stackrel{t\to 0}{\to}2\Lambda \Sigma+S \ \mathrm{in} \ L^1(\Vert \cdot \Vert_{\mathrm{HS}},\mathbb{P})\\
\mathrm{(iii)} & \frac{1}{s(t)}\E\left((W_t-W)(W_t-W)^{T}|W\right)\stackrel{t\to 0}{\to} T \ \mathrm{in} \ L^1(\Vert \cdot \Vert_{\mathrm{HS}},\mathbb{P})\\
\mathrm{(iv)} & \lim_{t\to 0}\frac{1}{s(t)}\E\left(|W_t-W|^2\mathbf{1}_{\lbrace |W_t-W|^2>\epsilon\rbrace}\right)=0,
\end{align*}
 
 \noindent for each $\epsilon>0$.
 
 \vspace{0.3cm}
Then
 
\begin{equation}
\mathcal{W}_1^{(d)}(W,\sqrt{\Sigma}Z)\leq \Vert\Lambda^{-1}\Vert_{\mathrm{op}}\left(\E|R|+\frac{1}{2\pi}\Vert\Sigma^{-\frac{1}{2}}\Vert_{\mathrm{op}}\E(\Vert S\Vert_{\mathrm{HS}}+\Vert T\Vert_{\mathrm{HS}})\right),
\end{equation}

\noindent where $\Vert\cdot \Vert_{\mathrm{op}}$ denotes the operator norm: for $A\in \C^{d\times d}$ 

\begin{equation}
\Vert A\Vert_{\mathrm{op}}=\sup_{x\in \C^d:|x|=1} |Ax|.
\end{equation}

\end{theorem}

\begin{remark}\label{rem:weaker}
As noted in \cite{ds}, we can replace the estimate for $\E(|W_t-W|^2\mathbf{1}_{|W_t-W|>\epsilon})$ by the weaker one 

\begin{equation}
\lim_{t\to 0}\frac{1}{s(t)}\E|W_t-W|^3=0
\end{equation}

\noindent since 

\begin{equation}
\E(|W_t-W|^2\mathbf{1}_{|W_t-W|>\epsilon})\leq \frac{1}{\epsilon}\E(|W_t-W|^3).
\end{equation}

\end{remark}

\subsection{Circular Dyson Brownian motion}

In this section we define circular Dyson Brownian motion and point out how it ties into our Stein's method argument.

\vspace{0.3cm}

Circular Dyson Brownian motion was introduced by Dyson \cite{dyson} and Discussed for example in \cite{spohn}. Its existence is proven in \cite{cl}. It is a model for diffusing particles confined to the unit circle and interacting with each other through a logarithmic repulsion. The main result of \cite{cl} is that one can make the following definition:

\begin{definition}[Circular Dyson Brownian motion]
Let $\beta >0$. $n$-dimensional Circular $\beta$-Dyson Brownian motion is a $C([0,\infty),\Delta_n)$ valued semimartingale (which we denote by $(x(t))_{t\geq 0}=(x_1(t),...,x_n(t))_{t\geq 0}$) which is the unique strong solution  to the system of stochastic differential equations

\begin{equation}\label{eq:dbm}
dx_j(t)=\frac{\beta}{2}\sum_{1\leq i\leq n, i\neq j}\cot \frac{x_j(t)-x_i(t)}{2}dt+\sqrt{2}db_j(t),
\end{equation}

\noindent for $j=1,...,n$. Here $b_j$ are i.i.d. standard Brownian motions.
\end{definition}

\begin{remark}
It is proven in \cite{cl}, that for $\beta\geq 1$ the particles almost surely do not collide (so $x_i(t)\neq x_j(t)$ for $i\neq j$ for all $t$), but for $\beta\in(0,1)$ they almost surely do.
\end{remark}

In the following remark we'll informally recall some basic facts from diffusion theory applied to our setting.

\begin{remark}\label{rem:stat}
As we are dealing with continuous semimartingales, we can make use of It\^o's lemma, and general facts from diffusion theory hold. In particular, a simple application of It\^o's lemma implies that we have for some fixed $x(0)\in \Delta_n$ and $C^2$ function $f$

\begin{equation}\label{eq:dynkin}
\E^{x(0)}\left[f(x(t))\right]=f(x(0))+\E^{x(0)}\left[\int_0^t [L_\beta f](x(s))ds\right],
\end{equation}

\noindent where $\E^{x(0)}$ denotes expectation with respect to the law of the process started from $x(0)$, and $L_\beta$ can be viewed as the infinitesimal generator of the process:

\begin{align}\label{eq:gene}
\notag L_\beta&=\frac{\beta}{2}\sum_{k=1}^n\sum_{l\neq k}\cot \frac{x_k-x_l}{2}\frac{\partial}{\partial x_k}+\frac{(\sqrt{2})^2}{2}\sum_{k=1}^n\frac{\partial^2}{\partial x_k^2}\\
&=\frac{\beta}{2}i\sum_{k=1}^n\sum_{l\neq k}\frac{e^{ix_k}+e^{ix_l}}{e^{ix_k}-e^{ix_l}}\frac{\partial}{\partial x_k}+\sum_{k=1}^n\frac{\partial^2}{\partial x_k^2}.
\end{align}

As for $\beta<1$ there can be collisions, there is some care to be taken about what the precise domain of the infinitesimal generator is (for example, if $f\in C^2(\Delta_n)$ is a function in the domain of the generator, then one must have that $\lim_{x_{j+1}\to x_j}\cot [(x_{j+1}-x_j)/2](\partial_{j+1}-\partial_j)f$ is finite, or in other words, $\partial_{j+1}f(x)|_{x_{j+1}=x_j}=\partial_j f(x)|_{x_{j+1}=x_j}$).

From \eqref{eq:dynkin} we see that if $\rho_t(x;x(0))$ is the density of the law of $x(t)$ started at $x(0)$, then it satisfies the equation 

\begin{equation}
\partial_t \rho_t(x,x(0))=L_\beta^* \rho_t(x,x(0)),
\end{equation}

\noindent where $L_\beta^*$ is the adjoint of $L_\beta$:

\begin{equation}
L_\beta^* f=\sum_{k=1}^n \frac{\partial^2}{\partial x_k^2}f-\frac{\beta}{2}\sum_{k=1}^n\sum_{l\neq k}\frac{\partial}{\partial x_k}\left[\cot\left(\frac{x_k-x_l}{2}\right)f\right].
\end{equation}

This implies that the C$\beta$E is a stationary distribution for circular Dyson Brownian motion. To see this, note that for 

\begin{align}
\rho(x)&=C_{n,\beta}\prod_{j<k}|e^{ix_j}-e^{ix_k}|^\beta\\
\notag &=C_{n,\beta}e^{-\beta \sum_{j<k}V(x_j-x_k)},
\end{align}

\noindent where $C_{n,\beta}$ is a normalization constant, and $V(x)=-\log|2\sin \frac{x}{2}|$, a simple calculation making use of the fact that $\frac{1}{2}\cot(x/2)=-V'(x)$ shows that 

\begin{equation}
L_\beta^*\rho=0.
\end{equation}

Thus the unique solution to $\partial_t\rho_t(x,x(0))=L_\beta^*\rho_t(x,x(0))$ with initial data given by the C$\beta$E: $\rho_0(x,x(0))=\rho(x)$, is $\rho_t(x,x(0))=\rho(x)$ - or the C$\beta$E is a stationary distribution for circular Dyson Brownian motion. 

A similar argument shows that for $g$ in the domain of $L_\beta$, $L_\beta^*[g \rho]=[L_\beta g]\rho$, i.e. that $L_\beta$ is in fact self-adjoint on $L^2(\rho)$. Thus the C$\beta$E is a reversible measure for Dyson Brownian motion which implies that the pair $(x(0),x(t))$ is exchangeable (i.e. $(x(0),x(t))\stackrel{d}{=}(x(t),x(0))$) for each $t>0$.
\end{remark}

Let us now prove our main estimates required for applying Theorem \ref{th:stein}. This entails estimating $\E(f(x(t))|x(0))$, when $f$ is a function relevant to Theorem \ref{th:stein}. This will be done through estimates on $L_\beta f$ for relevant $f$. For $\beta=2$ \cite{fulman,ds} make use of similar results for the heat kernel of the unitary group found in \cite{rains,levy} with a different kind of approach.

\begin{lemma}\label{le:powersum}
Let $x(0)$ be distributed according to the C$\beta$E and independent of $(b_j(t))$. Also let $k\in \Z$ and write for $x\in [0,2\pi]^n$, $p_k(x)=\sum_{j=1}^n e^{ikx_j}$. Then 

\begin{align}
\frac{1}{t}\left[\E^{x(0)}[p_k(x(t))]-p_k(x(0))\right]&\stackrel{t\to 0}{\rightarrow}[L_\beta p_k](x(0)),
\end{align} 

\noindent and 

\begin{equation}
L_\beta p_k=- n\frac{\beta}{2}|k|p_k-\left(1-\frac{\beta}{2}\right)k^2 p_k-\frac{\beta}{2}|k|\sum_{l=1}^{|k|-1}p_{\mathrm{sgn}(k)l}p_{\mathrm{sgn}(k)(|k|-l)},
\end{equation}

\noindent where $\mathrm{sgn}(k)=k/|k|$ for $k\neq 0$ and $L_\beta$ is the operator from \eqref{eq:gene}.

Moreover, for $k,l\in \Z$

\begin{align}
\frac{1}{t}\left[\E^{x(0)}[p_k(x(t))p_l(x(t))]-p_k(x(0))p_l(x(0))\right]&\stackrel{t\to 0}{\rightarrow}[L_\beta (p_k p_l)](x(0)),
\end{align} 

\noindent and 

\begin{equation}
L_\beta (p_k p_l)=p_kL_\beta p_l+p_l L_\beta p_k -2kl p_{k+l}.
\end{equation}

In both cases, the convergence is in $L^1$ with respect to the law of the C$\beta$E.

\end{lemma}

\begin{proof}
Let us first establish the claims about the action of $L_\beta$ on $p_k$ and $p_kp_l$. We have from \eqref{eq:gene}

\begin{align}
L_\beta p_k(x)&=-k^2p_k(x)+\frac{\beta}{2}i\sum_{l=1}^n\sum_{m\neq l}\frac{e^{ix_m}+e^{ix_l}}{e^{ix_m}-e^{ix_l}}ik e^{ikx_m}.
\end{align} 

Then note that 

\begin{align}
\mathcal{S}_k&:=\sum_{m=1}^n\sum_{l\neq m}\frac{e^{ix_m}+e^{ix_l}}{e^{ix_m}-e^{ix_l}}e^{ikx_m}\\
\notag &=\sum_{m=1}^n\sum_{l\neq m}(e^{ix_m}+e^{ix_l})\frac{e^{ikx_m}-e^{ikx_l}}{e^{ix_m}-e^{ix_l}}+\sum_{m=1}^n\sum_{l\neq m}(e^{ix_m}+e^{ix_l})\frac{e^{ikx_l}}{e^{ix_m}-e^{ix_l}}\\
\notag &=\sum_{m=1}^n\sum_{l\neq m}(e^{ix_m}+e^{ix_l})\frac{e^{ikx_m}-e^{ikx_l}}{e^{ix_m}-e^{ix_l}}-\mathcal{S}_k.
\end{align}

We thus conclude that 

\begin{align}
L_\beta p_k(x)&=-k^2p_k(x)-\frac{\beta}{4}k\sum_{m=1}^n\sum_{l\neq m}(e^{ix_m}+e^{ix_l})\frac{e^{ikx_m}-e^{ikx_l}}{e^{ix_m}-e^{ix_l}}.
\end{align}

For $k\in \Z_+$, we expand the difference quotient and find (using for example $p_0(x)=n$)

\begin{align}
L_\beta p_k(x)&=-k^2p_k(x)-\frac{\beta}{4}k\sum_{m=1}^n\sum_{l\neq m}(e^{ix_m}+e^{ix_l})\sum_{j=0}^{k-1}e^{ijx_m}e^{i(k-1-j)x_l}\\
\notag &=-k^2p_k(x)-\frac{\beta}{4}k\sum_{m=1}^n\sum_{l=1}^n(e^{ix_m}+e^{ix_l})\sum_{j=0}^{k-1}e^{ijx_m}e^{i(k-1-j)x_l}\\
\notag & \qquad +\frac{\beta}{4}k\sum_{m=1}^n 2e^{ix_m}k e^{i(k-1)x_m}\\
\notag &=-\left[1-\frac{\beta}{2}\right]k^2 p_k(x)-\frac{\beta}{4}k\sum_{j=0}^{k-1}\left[p_{j+1}(x)p_{k-1-j}(x)+p_j(x)p_{k-j}(x)\right]\\
\notag &=-\left[1-\frac{\beta}{2}\right]k^2 p_k(x)-\frac{\beta}{2}k\sum_{j=1}^{k-1}p_j(x)p_{k-j}(x)-\frac{\beta}{2}kn p_k(x),
\end{align}

\noindent which was the claim for $k>0$. For $k=0$ the claim is clear, and for $k<0$ it follows by complex conjugating the $k>0$ case. For calculating $L_\beta [p_k p_l]$, we note that if we write $\Delta=\sum_{j=1}^n \frac{\partial^2}{\partial x_j^2}$, then in general for twice differentiable functions $f$ and $g$ one has

\begin{equation}
\Delta [f g]=f\Delta g+g\Delta f +2\sum_{j=1}^n\left[\frac{\partial}{\partial x_j}f\right]\left[\frac{\partial}{\partial x_j}g\right].
\end{equation}

The first order part of $L_\beta$ satisfies a normal product rule so we find

\begin{align}
L_\beta[p_k p_l]&=p_k L_\beta p_l+p_l L_\beta p_k+2\sum_{j=1}^{n}\left[\frac{\partial}{\partial x_j}p_k\right]\left[\frac{\partial}{\partial x_j}p_l\right]\\
\notag &=p_k L_\beta p_l+p_l L_\beta p_k+2\sum_{j=1}^{n}\left[\frac{\partial}{\partial x_j}p_k\right]\left[\frac{\partial}{\partial x_j}p_l\right]\\
\notag &=p_k L_\beta p_l+p_l L_\beta p_k-2kl\sum_{j=1}^{n}e^{ikx_j}e^{ilx_j}\\
\notag &=p_k L_\beta p_l+p_l L_\beta p_k-2klp_{k+l}
\end{align}

\noindent which was the claim concerning the action of $L_\beta$ on $p_kp_l$. Let us now turn to the convergence part. From \eqref{eq:dynkin} we find that for any fixed $x\in \Delta_n$

\begin{align}
\left|\frac{\E^{x}\left[p_k(x(t))\right]-p_k(x)}{t}- [L_\beta p_k](x)\right|\leq \E^x\frac{\int_0^t\left|[L_\beta p_k](x(s))-[L_\beta p_k](x)\right|ds}{t}.
\end{align}

As we've seen that $L_\beta p_k$ is a polynomial in the variables $e^{ix_j}$, we see that $\sup_{x}|[L_\beta p_k](x)|$ is a finite number depending on $n$ and $k$. Thus the random variable we're taking an expectation of on the right side of the equation is uniformly bounded in $t$, and by the continuity of $t\mapsto x(t)$ at $t=0$, it converges to zero almost surely as $t\to 0$. Thus by the dominated convergence theorem (applied to the $\E^x$ integral) we conclude that the left side of the equation tends to zero. The same argument implies that the left side of the equation is bounded in $x$ by a constant depending only on $n$ and $k$, if we integrate over $x$ with respect to the law of the C$\beta$E, we can apply the dominated convergence theorem again to achieve $L^1$ convergence with respect to the law of the C$\beta$E. The argument for the $L^1$ convergence of the $p_k p_l$-term is similar.

\end{proof}

\section{Moment estimates of power sums for the C$\beta$E}

Before checking the conditions for Theorem \ref{th:stein}, we need some moment estimates on power sums. We need a simplified version of the main result in \cite{jm} (their results are analogous to those of \cite{diacshah} though extended to general $\beta$ from the unitary case through Jack polynomial theory):

\begin{theorem}[Jiang and Matsumoto]\label{th:jm}
Let $0\leq m\leq n$,

\begin{equation}
\begin{array}{ccc}
A= \left(1-\frac{\left|\frac{2}{\beta}-1\right|}{n-m+\frac{2}{\beta}}\mathbf{1}(\beta\leq 2)\right)^{m}, & \mathrm{and} & B= \left(1+\frac{\left|\frac{2}{\beta}-1\right|}{n-m+\frac{2}{\beta}}\mathbf{1}(\beta>2)\right)^{m}
\end{array}.
\end{equation}

Then

\begin{equation}
\E(|p_m(x)|^{2})\leq B\frac{2}{\beta}m
\end{equation}

\noindent and for $0\leq m\leq n$ with $0\leq j,k\leq m$, 

\begin{align}
\notag &\left|\E(p_{j}(x)p_{m-j}(x)p_{-k}(x)p_{k-m}(x))\right|\\
&\leq \begin{cases}
\max\lbrace |A-1|,|B-1|\rbrace\left(\frac{2}{\beta}\right)^{2}2\sqrt{j(m-j)k(m-k)} & k\neq j\\
B \left(\frac{2}{\beta}\right)^{2}2j(m-j), k=j.
\end{cases}
\end{align}
\end{theorem}

In most of our applications, we will have $m=\mathit{o}(n)$ and this becomes

\begin{corollary}\label{cor:moments}
For $0\leq m=\mathit{o}(n)$ and $n$ large enough

\begin{equation}
E(|p_m(x)|^{2})\leq 2\frac{2}{\beta}m
\end{equation}

\noindent and for $0\leq j,k\leq m$,

\begin{equation}
\left|\E(p_{j}(x)p_{m-j}(x)p_{-k}(x)p_{k-m}(x))\right|\leq \begin{cases}
\sqrt{j(m-j)k(m-k)}\mathcal{O}\left(\frac{m}{n}\right) & k\neq j\\
4 \left(\frac{2}{\beta}\right)^{2}j(m-j), k=j.
\end{cases}
\end{equation}

\end{corollary}

\begin{proof}
This follows directly from the definition of $A$ and $B$ noting that for $m=\mathit{o}(n)$

\begin{equation}
A,B=1+\mathcal{O}\left(\frac{m}{n}\right).
\end{equation}
\end{proof}

\section{Proof of Theorem \ref{th:main}}

We can now check the conditions required for Theorem \ref{th:stein} and make the estimates needed for the proof of Theorem \ref{th:main}. Let us begin by checking the conditions for Theorem \ref{th:stein}.

\vspace{0.3cm}

Let us write $W=T_d$, $x(t)=(x_1(t),...,x_n(t))$ for the $n$-dimensional circular $\beta$-Dyson Brownian motion started from an independent C$\beta$E$(n)$ vector $x=(x_1,...,x_n)$, and $W_t=(p_1(x(t)),...,p_d(x(t)))$. 

\vspace{0.3cm}

The first condition for Theorem \ref{th:stein} involved the conditional expectation of $W_t$ given $W$:

\subsection{$\E(W_t-W|W)$ as $t\to 0$} By Lemma \ref{le:powersum}, we have as $t\to 0$

\begin{align}
\lim_{t\to 0}\frac{1}{t}\E(W_t-W|W)&= (L_\beta p_1(x),...,L_\beta p_d(x))\\
&=-\Lambda W+R,
\end{align}

\noindent where  $\Lambda\in \C^{d\times d}$ with entries

\begin{equation}
\Lambda_{k,l}=\delta_{k,l}nk\frac{\beta}{2},
\end{equation}

\noindent  and $R\in \C^{d}$ with entries 

\begin{equation}
R_k=-k^{2}\left(\frac{\beta}{2}-1\right)p_k(x)-k\frac{\beta}{2}\sum_{l=1}^{k-1}p_l(x)p_{k-l}(x).
\end{equation}

Next we need $\E((W_t-W)(W_t-W)^{*}|W)$ as $t\to 0$.

\subsection{$\E((W_t-W)(W_t-W)^{*}|W)$ as $t\to 0$} 

For this, we need 

\begin{equation}
\E((p_j(x(t))-p_j(x))(p_{-k}(x(t))-p_{-k}(x))|x)
\end{equation}

\noindent for $j,k\in \Z_+$. To calculate this, we expand the product and consider each term separately:

\begin{equation}
\E(p_j(x(t))p_{-k}(x(t))|x)=p_j(x)p_{-k}(x)+t L_\beta(p_j(x) p_{-k}(x))+\mathit{o}(t),
\end{equation}

\begin{equation}
\E(p_j(x(t))p_{-k}(x)|x)=p_j(x)p_{-k}(x)+t p_{-k}(x)L_\beta p_j(x)+\mathit{o}(t),
\end{equation}

\noindent and

\begin{equation}
\E(p_j(x)p_{-k}(x(t))|x)=p_j(x)p_{-k}(x)+tp_j(x)L_\beta p_{-k}(x)+\mathit{o}(t).
\end{equation}

Thus

\begin{align}
\notag \E&((p_j(x(t))-p_j(x))(p_{-k}(x(t))-p_{-k}(x))|x)\\
&=t\left(L_\beta(p_j(x)p_{-k}(x))-p_j(x)L_\beta p_{-k}(x)-p_{-k}(x)L_\beta p_j(x)\right)+\mathit{o}(t).
\end{align}

Making use of Lemma \ref{le:powersum}, we find 

\begin{equation}\label{eq:var}
\lim_{t\to 0}\frac{1}{t}\E((p_j(x(t))-p_j(x))(p_{-k}(x(t))-p_{-k}(x))|x)=2jk p_{j-k}(x).
\end{equation}

We then write this as 

\begin{equation}
\lim_{t\to 0}\frac{1}{t}\E((W_t-W)(W_t-W)^{*}|W)=2\Lambda \Sigma+S,
\end{equation}

\noindent where $\Sigma\in \C^{d\times d}$,

\begin{equation}
(\Lambda\Sigma)_{k,l}=\delta_{k,l} n k^{2},
\end{equation}

\noindent or in other words

\begin{equation}
\Sigma_{k,l}=\frac{2}{\beta}\delta_{k,l}k.
\end{equation}

Moreover $S\in \C^{d\times d}$ with entries

\begin{equation}
S_{k,l}=(1-\delta_{k,l})2kl p_{k-l}(x).
\end{equation}

\subsection{$\E((W_t-W)(W_t-W)^{T}|W)$ as $t\to 0$}

Here we need 

\begin{equation}
\E((p_j(x(t))-p_j(x))(p_{k}(x(t))-p_{k}(x))|x)
\end{equation}

\noindent and a similar argument yields

\begin{equation}
\lim_{t\to 0}\frac{1}{t}\E((p_j(x(t))-p_j(x))(p_{k}(x(t))-p_{k}(x))|x)=-2jk p_{j+k}(x)
\end{equation}

\noindent or (again with convergence in $L^1$)

\begin{equation}
\lim_{t\to 0}\frac{1}{t}\E((W_t-W)(W_t-W)^{T}|W)=T,
\end{equation}

\noindent with 

\begin{equation}
T_{jk}=-2jk p_{j+k}(x).
\end{equation}

\subsection{$\E(|W_t-W|^{3})$ as $t\to 0$}

Following Remark \ref{rem:weaker}, it is enough for us to estimate $\E|W_t-W|^3$ (which for a diffusion one would expect to behave as $t^{3/2}$, but we still outline an argument for checking it directly) which in turn we can bound from above by $\sqrt{\E|W_t-W|^2}\sqrt{\E|W_t-W|^4}$. Conditioning on $W$ and using \eqref{eq:var}, one finds $\E|W_t-W|^2=\mathcal{O}(t)$ as $t\to 0$ and using similar arguments (in particular, the fact $L_\beta (fg)=f L_\beta g+gL_\beta f+2\sum_j (\partial_j f)(\partial_j g)$ repeatedly) one finds $\E|W_t-W|^4=\mathit{o}(t)$, and

\begin{equation}
\lim_{t\to 0}\frac{1}{t}\E(|W_t-W|^{3})=0.
\end{equation}

\subsection{The Wasserstein-1 distance}

Thus the conditions for Theorem \ref{th:stein} are met ($\Lambda$ is invertible and $\Sigma$ positive definite) and we have

\begin{equation}\label{eq:wasserstein}
\mathcal{W}_1^{(d)}(T_d,\sqrt{\Sigma}Z)\leq ||\Lambda^{-1}||_{\mathrm{op}}\left(\E|R|+\frac{1}{2\pi}||\Sigma^{-\frac{1}{2}}||_{\mathrm{op}}\E(||S||_{\mathrm{HS}}+||T||_{\mathrm{HS}})\right),
\end{equation}

\noindent where $Z$ a $d$-dimensional vector of i.i.d. standard complex Gaussians, $||\cdot ||_{\mathrm{op}}$ denotes the operator norm (with respect to the underlying Euclidean norm), $|\cdot |$ the Euclidean norm, and $||\cdot||_{\mathrm{HS}}$ the Hilbert-Schmidt norm. Let us check what these quantities are. 

\vspace{0.3cm}

Recall that 

\begin{equation}
\Lambda_{k,l}=\delta_{k,l}nk\frac{\beta}{2}
\end{equation}

\noindent and 

\begin{equation}
\Sigma_{k,l}=\delta_{k,l}\frac{2}{\beta}k.
\end{equation}

Being diagonal matrices, we note that 

\begin{equation}
||\Lambda^{-1}||_{\mathrm{op}}=\max_k \Lambda_{kk}^{-1}=\frac{2}{\beta}\frac{1}{n}
\end{equation}

\noindent and

\begin{equation}
||\Sigma^{-\frac{1}{2}}||_{\mathrm{op}}=\sqrt{\frac{\beta}{2}}.
\end{equation}

We'll estimate $\E|R|$ by $\sqrt{\sum_k \E|R_k|^2}$ and recall that $R_k$ consisted of two types of terms $R_k=A_k+B_k$ for which we write $|R_k|^2\leq 2(|A_k|^2+|B_k|^2)$ and estimate these separately. We use a similar estimate for the Hilbert-Schmidt norms. More precisely, recalling the definition of $R$, $S$, and $T$ we have

\begin{align}\label{eq:r}
\E|R|&\leq  C(\beta)\sqrt{\sum_{k=1}^d k^4 \E |p_k(x)|^2+\sum_{k=1}^{d}k^{2}\sum_{l,j=1}^{k-1}\E (p_l(x)p_{k-l}(x)p_{-j}(x)p_{j-k}(x))},
\end{align}

\begin{equation}\label{eq:s}
\E||S||_{\mathrm{HS}}\leq \sqrt{\sum_{j,k=1}^{d}\E|S_{j,k}|^{2}}=\sqrt{\sum_{j,k=1}^{d}(1-\delta_{j,k})4k^{2}j^{2}\E|p_{k-j}(x)|^{2}},
\end{equation}

\noindent and

\begin{equation}\label{eq:t}
\E||T||_{\mathrm{HS}}\leq \sqrt{\sum_{j,k=1}^{d}4j^{2}k^{2}\E|p_{j+k}(x)|^{2}}.
\end{equation}

We then make use of Corollary \ref{cor:moments} to get bounds on these:

\begin{lemma}

For $d=\mathcal{O}(\sqrt{n})$

\begin{equation}
\E|R|=\mathcal{O}(d^{3}),
\end{equation}

\begin{equation}
\E||S||_{\mathrm{HS}}=\mathcal{O}\left(d^{\frac{7}{2}}\right),
\end{equation}

\noindent and

\begin{equation}
\E||T||_{\mathrm{HS}}=\mathcal{O}\left(d^{\frac{7}{2}}\right).
\end{equation}

\end{lemma}

\begin{proof}
Plugging Corollary \ref{cor:moments} into \eqref{eq:r}, we find

\begin{equation}
\E|R|\leq C(\beta)\sqrt{\sum_{k=1}^dk^5+ \sum_{k=1}^{d}k^{2}\left(\sum_{j=1}^{k-1}j(k-j)+\frac{k}{n}\sum_{l\neq j}\sqrt{j(k-j)l(k-l)}\right)}.
\end{equation}

The first sum is  of order $d^6$. For the second sum, we note that 

\begin{equation}
\sum_{j=1}^{k-1}j(k-j)=k\sum_{j=1}^{k-1}j-\sum_{j=1}^{k-1}j^{2}=\mathcal{O}(k^{3}).
\end{equation}

For the third sum, we note that

\begin{equation}
\sum_{l\neq j}\sqrt{j(k-j)l(k-l)}\leq \left(\sum_{j=1}^{k-1}\sqrt{j(k-j)}\right)^{2}
\end{equation}

\noindent and 

\begin{align}
\notag \sum_{j=1}^{k-1}\sqrt{j(k-j)}&=\sqrt{k}\sum_{j=1}^{k-1}\sqrt{j}\sqrt{1-\frac{j}{k}}\\
&\leq \sqrt{k}\sum_{j=1}^{k-1}\sqrt{j}\left(1-\frac{1}{2}\frac{j}{k}\right)\\
\notag &=\mathcal{O}(k^{2}).
\end{align}

Thus

\begin{equation}
\E|R|\leq C\sqrt{d^6+\sum_{k=1}^{d}\frac{1}{n}k^{7}}.
\end{equation}

As $k=\mathcal{O}(\sqrt{n})$ 

\begin{equation}
\E|R|=\mathcal{O}(d^{3}).
\end{equation}

For $S$ we find (plugging Corollary \ref{cor:moments} into \eqref{eq:s} and using similar arguments as for $R$)

\begin{align}
\notag \E||S||_{\mathrm{HS}}&\leq \sqrt{\sum_{1\leq j<k\leq d}8 j^{2}k^{2}\frac{2}{\beta}(k-j)}\\
&\leq C\sqrt{\sum_{k=1}^{d}k^{6}}\\
&\notag =\mathcal{O}\left(d^{\frac{7}{2}}\right)
\end{align}

\noindent and in a similar manner

\begin{equation}
\E||T||_{\mathrm{HS}}=\mathcal{O}\left(d^{\frac{7}{2}}\right).
\end{equation}

\end{proof}

Noting that $\sqrt{\Sigma}Z=G_d$ and recalling that $||\Lambda^{-1}||_{op}=\mathcal{O}(n^{-1})$, plugging this into \eqref{eq:wasserstein} gives for $d=\mathit{o}(n^{2/7})$ 

\begin{equation}
\mathcal{W}_1^{(d)}(T_d,G_d)=\mathcal{O}(d^{7/2}/n).
\end{equation}

\noindent and Theorem \ref{th:main} is proven.

\section{The logarithm of the characteristic polynomial of the C$\beta$E}

One of the results proven in \cite{hko} is a limit theorem where they prove, using mainly results of \cite{diacshah}, that in a suitable Sobolev space of distributions, the real and imaginary parts of the logarithm of the characteristic polynomial of the CUE converge jointly in law to a pair of log-correlated Gaussian fields (in fact they can be understood as a restriction of the two-dimensional Gaussian Free Field restricted to the unit circle with a suitable convention for the "zero mode").

\vspace{0.3cm}

As the results in \cite{jm} generalize those of \cite{diacshah} to $\beta\neq 2$, one can prove a similar result for the characteristic polynomial of the C$\beta$E, though the estimates aren't quite as strong for the $\beta\neq 2$ case so one does not have quite as good a control on the roughness of the field - one needs to consider slightly larger Sobolev spaces than in the $\beta=2$ case. We'll give a brief argument for a proof of this fact here. First we recall the definition of the relevant Sobolev spaces.

\begin{definition}
For $s\in \R$ let 

\begin{equation}
\mathcal{H}_s=\left\lbrace (f_k)_{k\in \Z}: \sum_{k\in \Z}|f_k|^2(1+k^2)^s\right\rbrace
\end{equation}

\noindent and equip it with the inner product (we write $f=(f_k)_{k\in \Z}$ and $g=(g_k)_{k\in \Z}$)

\begin{equation}
\langle f,g\rangle_s=\sum_{k\in \Z}(1+k^2)^s f_k g_k^*.
\end{equation}

With this inner product, $\mathcal{H}_s$ is a separable Hilbert space. We denote by $\Vert\cdot \Vert_s$ the corresponding norm.
\end{definition}

\begin{remark}
For $s>0$, $\mathcal{H}_s$ can be interpreted as a subspace of the square integrable functions on the unit circle with $s$ describing the degree of smoothness of the functions. For $s<0$, $\mathcal{H}_s$ can be interpreted as the dual space of $\mathcal{H}_{-s}$ so it is a space of distributions. The quantities $f_k$ are interpreted as the Fourier coefficients of a function (or distribution) $f$.
\end{remark}

We then define our limiting object:

\begin{definition}
Let $(Z_j)_{j=1}^\infty$ be i.i.d. standard complex Gaussians and write formally

\begin{equation}
X(\theta)=\frac{1}{2}\sum_{j=1}^\infty \frac{1}{\sqrt{j}}(Z_j e^{-ij\theta}+Z_j^*e^{ij\theta}).
\end{equation}
\end{definition}

\begin{remark}
One can check that for any $\epsilon>0$, the above series converges almost surely in $\mathcal{H}_{-\epsilon}$ so $X$ can be understood as an element of $\mathcal{H}_{-\epsilon}$.
\end{remark}

We can now state the relevant limit theorem whose proof is essentially identical to that in \cite{hko}. A similar argument also appears in \cite{fks} so we give only a brief proof.

\begin{theorem}
Let $s>1/2$, $\beta>0$, and

\begin{equation}
P_n(\theta)=\prod_{j=1}^n (1-e^{i(x_j-\theta)}),
\end{equation}

\noindent where $(e^{ix_j})_{j=1}^n$ is distributed according to the $\mathrm{C}\beta\mathrm{E}(n)$. Moreover, let 

\begin{equation}
\begin{array}{ccc}
X_n(\theta)=\mathrm{Re}\log P_n(\theta) & \mathrm{and} &  Y_n(\theta)=\mathrm{Im}\log P_n(\theta),
\end{array}
\end{equation}

\noindent where the branch of $\log$ is such that $\mathrm{Im}\log (1-e^{i(x_j-\theta)})\in(-\pi/2,\pi/2]$ for all $j$.

\vspace{0.3cm}

Then $(X_n,Y_n)$ converges in law in $\mathcal{H}_{-s}\times \mathcal{H}_{-s}$ to $(\sqrt{2/\beta}X,\sqrt{2/\beta}Y)$ where $X$ is the field defined above and 

\begin{equation}
Y(\theta)=\frac{1}{2}\sum_{j=1}^\infty \frac{i}{\sqrt{j}}(-Z_j e^{-ij\theta}+Z_j^* e^{ij\theta}).
\end{equation}

\end{theorem}

\begin{proof}
Following \cite{hko}, we begin with the remark that expanding the logarithm gives (as an element of $\mathcal{H}_{-\epsilon}$)

\begin{equation}
\log P_n(\theta)=-\sum_{j=1}^\infty \frac{1}{j}\left(\sum_{k=1}^n e^{ij x_k}\right)e^{-ij\theta}=-\sum_{j=1}^\infty \frac{1}{j}p_j(x) e^{-ij\theta}.
\end{equation}

This implies that

\begin{equation}
X_n(\theta)=\frac{1}{2}\sum_{j=1}^\infty \frac{1}{\sqrt{j}}\left(-\frac{p_j(x)}{\sqrt{j}}e^{-ij\theta}-\frac{p_{-j}(x)}{\sqrt{j}}e^{ij\theta}\right)
\end{equation}

\noindent and 

\begin{equation}
Y_n(\theta)=\frac{1}{2 i}\sum_{j=1}^\infty \frac{1}{\sqrt{j}}\left(-\frac{p_j(x)}{\sqrt{j}}e^{-ij\theta}+\frac{p_{-j}(x)}{\sqrt{j}}e^{ij\theta}\right)
\end{equation}

\vspace{0.3cm}

Thus in the Fourier basis, we have convergence in the sense of finite dimensional distributions (as \cite{jm} or Theorem \ref{th:main} imply the convergence of say finite collections of the Fourier coefficients). 

\vspace{0.3cm}

By Prokohorov's theorem, to prove convergence it is then enough to prove tightness. For this, one uses the fact that the unit ball in $\mathcal{H}_{-s'}$ is compact in $\mathcal{H}_{-s}$ for $0<s'<s$. Let us then note that by Theorem \ref{th:jm}, if we take some small $\epsilon\in(0,1)$, there exists a constant $C$ such that  for $0\leq j\leq (1-\epsilon)n$, $\E|p_j(x)|^2\leq Cj$ while for $j\geq (1-\epsilon)n$ we trivially have $\E|p_j(x)|^2\leq n^2$.

\vspace{0.3cm}

Thus picking $s'\in(1/2,s)$ we have 

\begin{align}
\notag \E\Vert X_n\Vert_{-s'}^2&=\sum_{j=1}^\infty \frac{1}{j^2}(1+j^2)^{-s'}\E|p_j(x)|^2\\
&\leq C\sum_{1\leq j\leq (1-\epsilon)n}\frac{1}{j^{1+2s'}}+n^2\sum_{j\geq (1-\epsilon)n}\frac{1}{j^{2+2s'}}.
\end{align}

\noindent This is bounded as the first sum converges as $n\to\infty$ and the second one is $\mathcal{O}(n^{1-2s'})$. A similar bound holds for $Y_n$. Tightness then follows from the compactness of the unit ball mentioned above, and Markov's inequality.
\end{proof}

An interesting question is could one use stronger results on the linear statistics to give a stronger sense for this convergence. Here we essentially only used convergence of finite collections of the linear statistics and in no way made use of the fact that the number of them may grow with $n$. If one were able to extend $d$ in Theorem \ref{th:main} from $\mathit{o}(n^{2/7})$ to something close to $n$, it seems conceivable that one could estimate for example the distance of the maximum of the field $X_n$ to the maximum of the truncation of the field $X$. Indeed, the superexponential rate of convergence for a single linear statistic proven in e.g. \cite{johansson2} suggests that our bounds are likely to be far from optimal so perhaps something like this could be possible.

\vspace{0.3cm}

This could be one way to try to prove a conjecture of Fyodorov and Keating in \cite{fk} (for $\beta=2$), where they conjectured that the maximum of the logarithm of the characteristic polynomial of a CUE matrix behaves essentially like the maximum of a log-correlated Gaussian field (see e.g. \cite{madaule,drz}). Another motivation for trying to improve this type of results would be to make better sense of the connection between random matrix theory and Gaussian Multiplicative Chaos (for a result in this direction, see \cite{webb} based on results in \cite{dik,ck}, and for a review and an elementary approach to Gaussian Multiplicative Chaos see \cite{rv,berestycki}). Currently proving such results relies heavily on the determinantal structure and Riemann-Hilbert arguments available only for $\beta=2$.

\end{document}